\documentclass[twoside, 12pt]{amsart}
\usepackage{latexsym}
\usepackage{amssymb,amsmath,amsopn}
\usepackage[dvips]{graphicx}   
\usepackage{color,epsfig}      

\newtheorem{thm}{Theorem}

\newtheorem{defn}{Definition}
\newtheorem{rem}{Remark}

\DeclareMathOperator{\grad}{grad}

\DeclareMathOperator{\inter}{int}

\renewcommand{\Re}{\mathbb R}

\begin{document}
\title[On global equilibria]{On global equilibria of finely discretized curves and surfaces}
\author[G. Domokos and Z. L\'angi]{G\'abor Domokos and Zsolt L\'angi}
\address{G\'abor Domokos, Dept. of Mechanics, Materials and Structures, Budapest University of Technology,
M\H uegyetem rakpart 1-3., Budapest, Hungary, 1111}
\email{domokos@iit.bme.hu}
\address{Zsolt L\'angi, Dept.\ of Geometry, Budapest University of
Technology and Economics, Budapest, Egry J\'ozsef u. 1., Hungary, 1111}
\email{zlangi@math.bme.hu}

\keywords{equilibrium, convex surface, Poincaré-Hopf formula, polyhedral approximation}
\subjclass[2010]{53A05, 53Z05}
\thanks{The authors gratefully acknowledge the support of the J\'anos Bolyai Research Scholarship of the Hungarian Academy of Sciences and support from OTKA grant 104601}

\begin{abstract}
In our earlier work \cite{DLS} we identified the types and numbers of static equilibrium points 
of solids arising from fine, equidistant $n$-discretrizations of smooth, convex surfaces.
We showed that such discretizations carry equilibrium points on two scales: the local scale corresponds to the discretization, the global scale to the original, smooth surface.
In \cite{DLS} we showed that as $n$ approaches  infinity, the number of local equilibria fluctuate around specific values which
we call the imaginary equilibrium indices associated with the approximated smooth surface.
Here we show how the number of global equilibria can be interpreted, defined
and computed on such discretizations. Our results are relevant from the point of view of natural pebble
surfaces, they admit a comparison between field data based on hand measurements and laboratory data based on 3D scans.
\end{abstract}

\maketitle

\section{Introduction}

Static equilibria of convex bodies correspond to the singularities of the
gradient vector field characterizing their surface.
The study of equilibria of rigid bodies is a classic chapter of mathematics and mechanics; 
initiated by Archimedes \cite{Archimedes1}, the theory was revived in modern times by the works
of Cayley \cite{Cayley} and Maxwell \cite{Maxwell} yielding results on
the global number of stationary points. Further generalization
by Poincar\'e and Hopf led to the Poincar\'e-Hopf Theorem \cite{Arnold} on topological invariants.
If applied to generic, convex bodies, represented by gradient fields defined on the sphere, this theorem 
states that the number $S$ of `sinks' (stable equilibria), the number $U$ of `sources' (unstable
equilibria) and the number $N$ of saddles always satisfy the equation
\begin{equation} \label{Poincare}
S+U-N=2.
\end{equation}
This formula, the so-called Poincar\'e-Hopf formula can be regarded as a generalization
of the well-known Euler's formula \cite{Euler}  for convex polyhedra.

We also mention results on polyhedra;
monostatic polyhedra (i.e. polyhedra with just $S=1$ stable equilibrium point) have been studied
in \cite{Heppes}, \cite{Conway},\cite{Dawson} and \cite{DawsonFinbow} and more recently in \cite{Reshetov}.

The total number $T$ of equlibria $(T= S+U+N)$ has also been in the focus of research.
In planar, homogeneous, convex bodies (rolling along their boundary
on a horizontal support), we have $T\geq 4$ \cite{Domokos1}.
However, convex homogeneous objects with $T=2$ exist in the three-dimensional space (cf. \cite{VarkonyiDomokos}).  Zamfirescu \cite{Zamfirescu} showed that
for \emph{typical} convex bodies, $T$ is infinite, suggesting that equilibria in abundant numbers may occur in physically relevant scenarios.

Natural pebbles exhibit similar behavior: their convex hull is a multi-faceted polyhedron $P$ carrying many static equilibria \cite{Domokosetal}
appearing in strongly localized \emph{flocks}. In \cite{DLS} we studied this phenomenon and showed that if $P$ is defined as a sufficiently dense,
equidistant discretization of a smooth surface $M$ then flocks are indeed strongly localized in the vicinity of (isolated) equilibria of $M$
and the number, type and geometrical arrangement of equilibrium points \emph{inside} any single flock can be
expressed by the principal curvatures and the distance to the center of gravity of $M$.

Here we make one further step: we give definitions and prove statements based on which, for polyhedra $P$ defined by sufficiently
dense discretizations, the number and type of maxima and minima of $M$ can be determined and the number of saddles may be obtained via equation
(\ref{Poincare}). These results, defining the \emph{number of flocks} on a dense discretization $P$ corresponding to a smooth surface $M$, 
may help to bridge the gap between hand experiments on pebbles,
identifying the equilibria of $M$ and 3D computer scans, identifying the equilibria of $P$ \cite{Domokosetal}.
We formulate our results for general functions in two variables, however, all results are valid for convex surfaces interpreted as the distance function
measured form the center of gravity. We also formulate the results for functions in one variable (2D convex shapes) where the statements are rather simple.

\section{Main results}

We start with some preliminary assumptions and definitions.

Let $f : [0,a] \times [0,b] \to \Re$ be a $C^3$-class function.
Consider a division $D_n$ of the rectangle $D=[0,a] \times [0,b]$ into $n \times n$ congruent rectangles.
We call the vertices of these rectangles \emph{grid vertices}, and denote the grid vertex $\left( \frac{i}{n}a, \frac{j}{n}b \right)$ by $p_{i,j}$.
The \emph{neighbors} of the grid vertex $p_{i,j}$ are the four grid vertices $p_{i \pm 1,j}$ and $p_{i,j \pm 1}$.
The two pairs $p_{i \pm 1,j}$ and $p_{i,j \pm 1}$ are called \emph{opposite neighbors} of $p_{i,j}$.

Recall that a point $p \in R$ is called a \emph{stationary} point of $f$, if $f'_x(p) = f'_y(p) = 0$.

\begin{defn}
A grid vertex $p$ is \emph{stationary}, if for any opposite pair $\{ q, q' \}$ of its neighbors, $f(p) \geq \max \{ f(q), f(q') \}$ or 
$f(p) \leq \min \{ f(q), f(q') \}$ is satisfied.
\end{defn}

If $p_{i,j}$ is a grid vertex, then the \emph{grid circle of centre $p_{i,j}$ and radius $r$} is the set
\[
C_r(p_{i,j}) = \{ p_{l,m} : \max \{ |l-i|, |m-j| \} \leq r \}.
\]

During the consideration, we assume that $f$ has finitely many stationary points, each in the interior of the domain $D$, and the determinant of the Hessian of $f$ at each of them is not zero.
We assume that the grids we use are nondegenerate; more specifically, that if $p \neq p'$ are two grid vertices, then $f(p) \neq f(p')$.

\begin{thm}\label{thm:auxiliary}
Let $p=(x_0,y_0) \in \inter D$.
\begin{enumerate}
\item[(1)] If $p$ is \emph{not} a stationary point of $f$, then $p$ has a neighborhood $U \subset D$ such that for any $n \geq 1$, if the grid vertex $p_{i,j}$ of $D_n$, and each of its neighbors, is contained in $U$, then $p_{i,j}$ is not a stationary grid vertex.
\item[(2)] If $p$ is a local minimum of $f$, then $p$ has a neighborhood $U$ and some suitable value of $r$ such that for every sufficiently large $n$, there is exactly one grid vertex $p_{i,j}$ of $D_n$ in $U$, which is minimal within its grid circle $C_r(p_{i,j})$. 
\item[(3)] If $p$ is a local maximum of $f$, then $p$ has a neighborhood $U$ and some suitable value of $r$ such that for every sufficiently large $n$, there is exactly one grid vertex $p_{i,j}$ of $D_n$ in $U$, which is maximal within its grid circle $C_r(p_{i,j})$.
\item[(4)] If $p$ is a saddle point of $f$, then $p$ has a neighborhood of and some suitable value of $r$ such that for every sufficiently large $n$, any grid vertex $p_{i,j}$ of $D_n$ in $U$ is neither a local maximum, nor a local minimum within its grid circle $C_r(p_{i,j})$.
\end{enumerate}
\end{thm}

\begin{proof}
First, we prove (1).
Let $L$ be the line through the origin, perpendicular to $\grad f(p)$. Note that the derivative of $f$ is zero in this direction.
Let $\varepsilon > 0$ be sufficiently small, and $A$ be the union of the lines, through $p$, the angles of which with $L$ is not greater than $\varepsilon$.
Note that by the continuity of $\grad f$, $p$ has a neighborhood $U$ such that for any $q \in U$, $\grad f (q)$ is perpendicular to some line in $A$.
This implies that if $q \in U$, and $A$ contains no line parallel to the vector $u$, then $f'_u(q) \neq 0$. Without loss of generality, we may assume that $U$
is a Euclidean disk in $\Re^2$.

Now, consider any division $D_n$, and assume that the grid vertex $p_{i,j}$ and all its neighbors are contained in $U$.
Since $\varepsilon > 0$ be sufficiently small, the $x$-axis, or the $y$-axis is not parallel to any line in $A$.
Without loss of generality, let the $x$-axis have this property.
We show that the sequence $f(p_{i-1,j}), f(p_{i,j})$ and $f(p_{i+1,j})$ is strictly monotonous.
Indeed, if, for example, $f(p_{i,j}) \geq \max \{ f(p_{i-1,j}), f(p_{i+1,j})\}$, then by the Lagrange Theorem, for some $q_1,q_2 \in U$, we have
$f'_x(q_1) \leq 0 \leq f'_x(q_2)$, which, by the continuity of $f_x$, yields that  for some $q \in U$, we have $f'_x(q) = 0$.
Nevertheless, it contradicts the definition of $A$.
If $f(p_{i,j}) \leq \min \{ f(p_{i-1,j}), f(p_{i+1,j})\}$, we can reach a contradiction in a similar way.
Thus, $p_{i,j}$ is not a stationary grid vertex.

In the next part, we prove (2). Without loss of generality, assume that $f(p)=0$.
Note that since $p$ is a local minimum, both eigenvalues $\lambda_1 \leq \lambda_2$ of the Hessian of $f$ at $p$ are positive.
Let $P_2$ denote the second order Taylor polynomial of $f$ around $p$.
Then $P_2$ is a quadratic form with eigenvalues $\frac{\lambda_1}{2} > 0$ and $\frac{\lambda_2}{2} > 0$, and the curve $P_2 = 1$ is an ellipse.
Now, since $f$ is $C^3$-class, there is some $\bar{L} \in \Re$ such that for every $(x,y) \in D$, we have
\[
| f(x,y)-P_2(x,y) | < \frac{\bar{L}}{\sqrt{2}} \left( |x|^3 + x^2 |y| + |x| y^2 + |y|^3 \right) = \frac{\bar{L}}{\sqrt{2}} \left( |x| + |y| \right) \left( x^2 + y^2 \right) \leq L \left( x^2 + y^2 \right)^{3/2},
\]
which yields that for some suitable $L \in \Re$, we have $| f(q)-P_2(q) | \leq \left( P_2(q) \right)^{3/2}$ for every $q \in D$.

Let $\varepsilon > 0$ be sufficiently small. Choose a neighborhood $U$ of $p$ such that
\begin{itemize}
\item for every $q \in U$, we have $f(q) > 0$, and $| f(q)-P_2(q)| < \varepsilon P_2(q)$,
\item $f$ is convex in $U$.
\end{itemize}
Observe that the second condition holds for any convex neighborhood of $p$, where the Hessian of $f$ has only positive eigenvalues, and, the existence of such a neighborhood follows from the fact that $f$ is $C^3$-class.
Now, since $P(q)$ is homogeneous, every point $q \in D$, with $f(q)=\alpha$, is contained between the ellipses $P_2(q) =  (1-\varepsilon) \alpha$ and $P_2(q) =  (1+\varepsilon) \alpha$. Note that if $\varepsilon$ is sufficiently small, for any value of $\alpha$ and any point $q$ of the level curve $f(x,y)=\alpha$, the angle between the two tangent lines of the ellipse $P_2(x,y) =  (1-\varepsilon) \alpha$, passing through $q$, is at least $\frac{\pi}{3}$.

Fix any division $D_n$, and consider the level curves $f(x,y) = \alpha$, as $\alpha \geq 0$ increases. Let $\bar{p}$ be the first grid vertex that reaches the boundary of such a curve. Clearly, $f(\bar{p}$ is minimal among all the grid vertices in $U$.
Let
\begin{equation}\label{eq:r}
r \geq \max \left\{ \frac{3\sqrt{a^2+b^2}}{2 \min \{ a,b\}}, \frac{\lambda_2}{\lambda_1} \frac{\sqrt{a^2+b^2}}{\min \{ a,b\}} \sqrt{\frac{1+\varepsilon}{1-\varepsilon}} \right\} ,
\end{equation}
where $\tau = \frac{\lambda_2}{\lambda_1} \geq 1$ is the ratio of the two eigenvalues of the Hessian of $f$ at $p$.
In the remaining part of the proof of (2), we show that there is no other grid vertex in $U$ which is minimal within its grid circle of radius $r$.

Assume, for contradiction, that the grid vertex $q$ is minimal within $C_r(q)$, and let $f(q)=\beta$. Then the level curve $f(x,y)=\beta$ already
contains some grid vertex $q'$ in its interior. Note that the semi-axes of the ellipse $P_2(x,y)=t$ are of length $\sqrt{\frac{2t}{\lambda_i}}$, where $i=1,2$.
Recall that the curve $f(x,y)=\beta$ is contained in the ellipse $P_2(x,y) = (1+\varepsilon) \beta$, and the diameter of the latter curve is $2\sqrt{\frac{2(1+\varepsilon)\beta}{\lambda_1}}$.
Since, according to our assumption, $q'$ is contained in the interior of $P_2(x,y) = (1+\varepsilon) \beta$, and $f(q') < f(q)$, we obtain that
\begin{equation}\label{eq:whatweknow}
r \delta < 2\sqrt{\frac{2(1+\varepsilon)\beta}{\lambda_1}},
\end{equation}
where $\delta = \min \left\{ \frac{a}{n}, \frac{b}{n} \right\}$ denotes the minimal distance between any two grid vertices.

Let $w$ be the point of $P_2(x,y) = (1-\varepsilon) \beta$ closest to $q$.
Let $\Delta = \frac{\sqrt{a^2+b^2}}{n} = \frac{\sqrt{a^2+b^2}}{\min \{ a,b\}} \delta$, and observe that any circle of diameter $\Delta$ contains a grid vertex.
We show that the circle $C$ of diameter $\Delta$, touching the ellipse $P_2(x,y) = (1-\varepsilon) \beta$ at $w$ from inside, is contained in the ellipse.
By Blaschke's Rolling Ball Theorem, to do this it suffices to show that $\frac{\Delta}{2}$ is not greater than any radius of curvature of the ellipse.
It is a well-known fact that the radius of curvature at any point of an ellipse with semi-axes $M \geq m$ is at least $\frac{m^2}{M}$ and at most $\frac{M^2}{m}$.
Thus, a simple computation yields that what we need to show is
\begin{equation}\label{eq:whatweneed}
\Delta \leq 2 \frac{\sqrt{2(1-\varepsilon) \beta \lambda_1}}{\lambda_2}.
\end{equation}
To show (\ref{eq:whatweneed}), we can combine (\ref{eq:whatweknow}) with the definition of $r$ in (\ref{eq:r}).

Let $\bar{C}$ be the circle of radius $\Delta$ that touches the tangent lines of the ellipse $P_2(x,y) = (1-\varepsilon) \beta$ through $q$. 
Since $f$ is convex in $U$, the level curve $f(x,y)=\beta$ is also convex, and thus, this circle is also contained inside the level curve $f(x,y)=\beta$.
On the other hand, $\bar{C}$ as any other circle of diameter $\Delta$, contains a grid vertex $q''$. Then, our previous observation yields that $f(q'') < \beta = f(q)$. To finish the proof, we show that $\bar{C}$ is contained in the circle of radius $r \delta$, centered at $q$, which implies that $q''$ is contained in the grid circle of radius $r$, centered at $q$.

Assume, for contradiction, that it is not so. Let $\phi$ be the angle between the two tangent lines of the ellipse $P_2(x,y)=(1-\varepsilon) \beta$, through $q$.
Since the angle between these two tangent lines is at least $\frac{\pi}{3}$, a simple computation yields that the distance of the centre of $\bar{C}$ and
$q$ is at most $\Delta$, and hence no point of $\bar{C}$ is farther from $q$ than $\frac{3}{2} \Delta = \frac{3\sqrt{a^2+b^2}}{2 \min \{ a,b\}} \delta \leq r \delta$, which finishes the proof of (2).

To prove (3), we can apply (2) for the function $-f$.

Finally, we prove (4). Let $f(p) = 0$. Then, in a neighborhood $U$ of $q$, the set $\{ f(q) = 0 \}$, $q \in U$ can be decomposed into the union of two $C^2$-class curves, crossing each other at $q$, and for any $\alpha \neq 0$, the set $\{ f(q) = \alpha \}$, $q \in U$ is the union of two disjoint, $C^2$-class curves.
Furthermore,  if $U$ is sufficiently small, there is some sufficiently small $\phi>0$ and $\varepsilon > 0$ such that for any $q \in U$
\begin{itemize}
\item there is a closed angular domain $A$ with apex $q$ and angle $\phi$ such that for any point $q' \in A$ with $0 < |q'-q| < \varepsilon$, we have $f(q) < f(q')$;
\item there is a closed angular domain $B$ with apex $q$ and angle $\phi$ such that for any point $q' \in B$ with $0 < |q'-q| < \varepsilon$, we have $f(q) > f(q')$.
\end{itemize}
Clearly, for a sufficiently large $r$ (chosen independently of $q$), any such closed angular domain in $U$ contains a vertex of $C_r(q)$, which yields the assertion.
\end{proof}

\begin{thm}\label{thm:main}
Le $f$ have $s$ local minima and $u$ local maxima. Then there is some $r$ such that for any sufficiently large $n$, exactly $s$ grid vertices of $D_n$ are minimal, and exactly $u$ grid vertices of $D_n$ are maximal within their grid circles of radius $r$.
\end{thm}

\begin{proof}
Fix some $r$ such that any stationary point $q$ of $f$ has some neighborhood that satisfies the corresponding conditions in (2), (3) or (4) of Theorem~\ref{thm:auxiliary}.
Observe that we can choose $\varepsilon_1,\varepsilon_2>0$ such that
\begin{itemize}
\item if $q$ is a stationary point, the assertion in (2), (3) or (4) Theorem~\ref{thm:auxiliary} holds in the $\varepsilon_1$-neighborhood $U_q$ of $q$;
\item if $q$ is not a stationary point, and its distance from any stationary point is at least $\varepsilon_1$, then (1) holds in the $\varepsilon_2$-neighborhood $U_q$ of $q$.
\end{itemize}

Now, let $n$ be large enough such that for any point $q \in D$, $C_r(q) \subset U(q)$, and for any stationary point, (2), (3) and (4) can be applied, and then, the theorem follows.
\end{proof}

\begin{rem}
Note that we may apply the following theorem for a parametrized convex surface $r=r(u,v)$, with $x,y$ as $(u,v)$, and $z=f(x,y)$ as the distance function $\| r(u,v)\|$.
\end{rem}

Note that the one-dimensional case of the problem is straightforward. More specifically, the following holds.

\begin{rem}\label{rem:2D}
Let $f : [a,b] \to \Re$ be a $C^2$-class function with finitely many stationary points, each in the interval $(a,b)$, such that the second derivative of $f$ at each such point is not zero. Let $x_1, x_2 ,\dots, x_k$ denote the local minima, and $x'_1, x'_2, \ldots, x'_l$ denote the local maxima of $f$. Let $P = \{ p_1,p_2,\ldots, p_{n-1} \}$ denote the set of the vertices of the equidistant $n$-element partition of $[a,b]$, contained in $(a,b)$, and assume that for any $i=0,1,\ldots,n-1$, $f(p_i) \neq f(p_{i+1})$. Then, if $n$ is sufficiently large, there are exactly $k$ local minima, and $l$ local maxima in $P$.  
\end{rem}

\section{Applications}

Our results may help to relate hand experiments on pebbles to the results of 3D computer scans. The latter identify the exact convex hull as  a multi-faceted polyhedron
(often with several thousand faces) and locate the equilibrium points of this polyhedron. As predicted by \cite{DLS}, these appear in flocks and the number and
type of equilibria observed \emph{inside} each flock appears to be well approximated by the numbers predicted in \cite{DLS}. Hand experiments, on the other hand,
tend to identify each flock as one single equilibrium point, associated with an (imaginary) smooth surface. In \cite{Domokosetal} we introduced a `fudge' parameter
$\mu$ describing the uncertainty of hand experiments, $\mu=0$ corresponding to the exact measurement which is identical to the computer output.

When plotting the number $T$ of equilibria versus $\mu$ we observed that after a steep initial drop, the plot has  a long plateau extending often over several orders of
magnitude of $\mu$ and the function value $T^{\star}$ of this plateau we associated with the number of equilibria of the (imaginary) smooth surface. Our current note
gives an independent definition for this number. Comparing those values may not only shed light on the applicability of our `fudge' parameter but also
may significantly contribute to the evaluation of geological field experiments.

\end{document}